\documentclass[a4paper,11pt]{article}
\usepackage[utf8]{inputenc}
\usepackage[T1]{fontenc}
\usepackage[english]{babel}
\usepackage[nomath]{lmodern}
\usepackage[margin=1in]{geometry}
\usepackage{microtype, tikz}

\usetikzlibrary{arrows.meta,positioning,shapes.geometric,fit}

\usepackage{graphicx}

\usepackage{amsthm,mathtools,amssymb}
\usepackage{hyperref}
\usepackage[capitalise]{cleveref}
\hypersetup{
  colorlinks=true,
  linkcolor=blue,
  citecolor=blue,
  urlcolor=blue,
  pdftitle={Short Proofs for the Minimum Number of Determined Circles},
  pdfauthor={Stijn Cambie, Denis Paliy, Rafal Wrona}
}

\newtheorem{lemma}{Lemma}

\newcommand{\cL}{\mathcal L}

\definecolor{richblue}{HTML}{2166AC}
\definecolor{richred}{HTML}{B2182B}
\definecolor{richgreen}{HTML}{1B7837}
\definecolor{richviolet}{HTML}{762A83}
\tikzset{
  rich1/.style={richblue,line width=.75pt},
  rich2/.style={richred,line width=.75pt,dashed},
  rich3/.style={richgreen,line width=.8pt,
    dash pattern=on 1.4pt off 1.1pt},
  rich4/.style={richviolet,line width=.75pt,dash dot},
  helper/.style={black!38,line width=.45pt,densely dashed},
  selected/.style={circle,fill=black,inner sep=1.35pt},
  panel/.style={draw=black!25,rounded corners=.8pt},
  every node/.style={font=\scriptsize}
}

\allowdisplaybreaks
\newtheorem{theorem}{Theorem}
\theoremstyle{remark}
\newtheorem*{relatedremark}{Remark}

\newcommand{\eightcoordinates}{%
  \coordinate (A) at (-1.5,{sqrt(7)/2});
  \coordinate (B) at (-.5,{sqrt(7)/2});
  \coordinate (C) at (.5,{sqrt(7)/2});
  \coordinate (D) at (1.5,{sqrt(7)/2});
  \coordinate (E) at ({(1-sqrt(7))/4},{(1+sqrt(7))/4});
  \coordinate (F) at ({-(1+sqrt(7))/4},{(sqrt(7)-1)/4});
  \coordinate (G) at ({(sqrt(7)-1)/4},{(sqrt(7)+1)/4});
  \coordinate (H) at ({(sqrt(7)+1)/4},{(sqrt(7)-1)/4});}

\newcommand{\eightpoints}{%
  \fill[selected] (A) circle (1.2pt) node[above left=-1pt] {$A$};
  \fill[selected] (B) circle (1.2pt) node[above=2pt] {$B$};
  \fill[selected] (C) circle (1.2pt) node[above=2pt] {$C$};
  \fill[selected] (D) circle (1.2pt) node[above right=-1pt] {$D$};
  \fill[selected] (E) circle (1.2pt) node[below left=-1pt] {$E$};
  \fill[selected] (F) circle (1.2pt) node[below left=-1pt] {$F$};
  \fill[selected] (G) circle (1.2pt) node[below right=-1pt] {$G$};
  \fill[selected] (H) circle (1.2pt) node[below right=-1pt] {$H$};}

\title{Short Proofs for the Minimum Number of Determined Circles}
\author{Stijn Cambie\thanks{Department of Computer Science, KU Leuven Campus Kulak-Kortrijk, 8500 Kortrijk, Belgium. Supported by a postdoctoral fellowship by the Research Foundation Flanders (FWO) with grant number 1225224N.}
\and Denis Paliy\and Rafa\l{} Wrona}
\date{11 October 2026}
\begin{document}
\maketitle

\begin{abstract}
We present short proofs for the main parts of the recently completed solution
of Erd\H{o}s Problem~506, which asks for the minimum number of distinct
circles determined by a finite planar point set that is neither
collinear nor concyclic. An inversion argument and a seven-point line
lemma show that eight points determine at least $17$ circles, and an
explicit configuration attains this bound. A single argument combining
Melchior's inequality with Langer's incidence bound establishes the
sharp formula for every $n\ge15$. The proofs give concise alternatives
to longer arguments in the recent literature; the remaining finite
cases are cited separately.
\end{abstract}

\section{Introduction}

The problem studied here goes back to Erd\H{o}s's 1961 collection of
unsolved problems~\cite{Erdos1961}. He asked for the least number of
distinct circles determined by $n$ points in the plane that do not all lie
on one circle. Taken literally, this allows all points to be collinear and
has the trivial answer zero. Following Elliott's treatment~\cite{Elliott},
we use the intended nondegenerate form: the points are neither all
collinear nor all concyclic.

Call such a point set \emph{admissible}. A noncollinear triple determines
a unique circle. Write $C(P)$ for the number of distinct circles determined
by $P$, counting each circle once even when it contains more than three
selected points, and put
\[
 f(n)=\min\{C(P): |P|=n\text{ and }P\text{ is admissible}\},\qquad n\ge4.
\]
The standard construction gives
\begin{equation}\label{eq:standard-upper}
 f(n)\le g(n):=1+\binom{n-1}{2}
 -\left\lfloor\frac{n-1}{2}\right\rfloor.
\end{equation}
Indeed, take a circle $\Gamma$, its centre $O$, and $n-1$ points on
$\Gamma$ arranged in $\lfloor(n-1)/2\rfloor$ antipodal pairs, with one
unmatched point when $n$ is even. Besides $\Gamma$, each nonantipodal pair
determines a circle with $O$, and different pairs give different circles.

We give short proofs of the main cases of the recently
completed classification.
\begin{theorem}\label{thm:main}
The minimum number of circles satisfies
\[
 f(8)=17,\qquad f(n)=g(n)\quad\text{for every }n\ge15.
\]
\end{theorem}

For $n=8$, our proof uses a seven-point incidence lemma after
inversion. For $n\ge15$, a weighted-line inequality gives a single
argument for the whole range. Both approaches were already developed
in related forms in Wrona's manuscript~\cite{Wrona}: a block-counting
and inversion proof for $n=8$, and a uniform incidence proof for
$n\ge15$. The present note streamlines the arguments; related
methods occur in the parallel work of Wang~\cite{Wang} and
Paliy~\cite{PaliyCorrected}. In particular, our proof for $n=8$
needs no computer-assisted classification of incidence types,
and the proof for $n\ge15$ uses one argument throughout.

Elliott's work initiated the study of circle counts~\cite{Elliott}.
Jucovi\v{c} established the small values for $4\le n\le6$ and
showed $f(7)\ge10$~\cite{Jucovic}. Purdy and Smith corrected the
large-$n$ argument of Elliott, proving $f(n)=g(n)$ for
$n\ge394$~\cite{PurdySmith}. The recent investigations
\cite{Wrona,Wang,PaliyCorrected} close the remaining gap.

The complete classification, combining the classical
cases $4\le n\le6$~\cite{Jucovic} with recent contributions
\cite{WhiteClaude,Wrona,Wang,PaliyCorrected},
is
\begin{equation}\label{eq:full-classification}
 f(n)=\begin{cases}
  8,&n=6,\\
  11,&n=7,\\
  17,&n=8,\\
  g(n),&n=4,5\text{ or }n\ge9.
 \end{cases}
\end{equation}
The cases $n=7$ and $9\le n\le14$ are not proved here; see
\cite{WhiteClaude,Wrona,Wang,PaliyCorrected}. The latter require
additional finite-case analysis and computational verification.
Thus \cref{thm:main} is proved below, whereas the full
classification~\eqref{eq:full-classification} also uses those results.

Wrona's repository~\cite{Wrona} contains a Lean~4 formalization of
the full classification (theorem \texttt{Erdos506.erdos\_506}); its
published axiom audit reports only the standard axioms
\texttt{propext}, \texttt{Classical.choice}, and \texttt{Quot.sound}.
The accompanying Lean files of~\cite{Wang} formalize an arithmetic
layer of that proof. Neither formalization is needed to follow the
mathematical arguments below.
Extremal configurations are illustrated in~\cref{fig:optimalConf}.

\begin{figure}[t]
    \centering
    
\begin{tikzpicture}[
  point/.style={circle,fill=black,inner sep=1.7pt},
  rich/.style={line width=.75pt},
  auxiliary/.style={
    line width=.65pt,
    densely dashed,
    gray!75
  },
  every node/.style={font=\small}, scale=0.8
]
  \pgfmathsetmacro{\q}{sqrt(3)}
  \pgfmathsetmacro{\u}{sqrt(7)}

  \begin{scope}[xshift=1.7cm]
    \coordinate (A6) at (1,0);
    \coordinate (B6) at (-.5,{\q/2});
    \coordinate (C6) at (-.5,{-\q/2});
    \coordinate (D6) at (1,\q);
    \coordinate (E6) at (-2,0);
    \coordinate (F6) at (1,-\q);

    \draw[rich] (D6)--(E6)--(F6)--cycle;

    \foreach \P/\lab/\pos in {
      A6/A/right,
      B6/B/above left,
      C6/C/below left,
      D6/D/above,
      E6/E/left,
      F6/F/below
    }
      \node[point,label=\pos:{\(\lab\)}] at (\P) {};

    \node[font=\normalsize] at (-.5,-2.25) {\(n=6\)};
  \end{scope}

  \begin{scope}[xshift=6.4cm]
    \coordinate (A) at (1,0);
    \coordinate (B) at (-.5,{\q/2});
    \coordinate (C) at (-.5,{-\q/2});
    \coordinate (D) at (1,\q);
    \coordinate (E) at (-2,0);
    \coordinate (F) at (1,-\q);
    \coordinate (O) at (0,0);

    \draw[rich] (D)--(E)--(F)--cycle;
    \draw[auxiliary] (D)--(C) (E)--(A) (F)--(B);

    \foreach \P/\pos in {
      A/right,
      B/above left,
      C/below left,
      D/above,
      E/left,
      F/below,
      O/above right
    }
      \node[point,label=\pos:{\(\P\)}] at (\P) {};

    \node[font=\normalsize] at (-.5,-2.25) {\(n=7\)};
  \end{scope}

  \begin{scope}[
    xshift=10.6cm,
    yshift=-.75cm,
    scale=1.38,
    every node/.append style={transform shape}
  ]
    \coordinate (P0) at (-1.5,{\u/2});
    \coordinate (P1) at (-.5,{\u/2});
    \coordinate (P2) at (.5,{\u/2});
    \coordinate (P3) at (1.5,{\u/2});
    \coordinate (P4) at ({(1-\u)/4},{(1+\u)/4});
    \coordinate (P5) at ({-(1+\u)/4},{(\u-1)/4});
    \coordinate (P6) at ({(\u-1)/4},{(\u+1)/4});
    \coordinate (P7) at ({(\u+1)/4},{(\u-1)/4});

    \draw[rich] (P0)--(P3);
    \draw[auxiliary] (P0)--(P7) (P3)--(P5);

    \foreach \P/\lab in {P0/A,P1/B,P2/C,P3/D}
      \node[point,label=above:{\(\lab\)}] at (\P) {};

    \foreach \P/\lab/\pos in {
      P4/E/left,
      P5/F/below,
      P6/G/right,
      P7/H/below
    }
      \node[point,label=\pos:{\(\lab\)}] at (\P) {};

    \node[font=\normalsize] at (0,-.55) {\(n=8\)};
  \end{scope}

\end{tikzpicture}\quad 
\begin{tikzpicture}
  \clip (-1.95,-1.95) rectangle (1.95,1.95);
  \draw[rich1] (0,0) circle[radius=1.35];
  \foreach \a in {0,45,90,135}
    \draw[helper] (\a:1.35)--({\a+180}:1.35);
  \foreach \a/\lab/\where in {
    90/A/above,45/B/above right,0/C/right,315/D/below right,
    270/E/below,225/F/below left,180/G/left,135/H/above left}
    \fill[selected] (\a:1.35) circle (2.5pt)
      node[\where=2pt] {\large \(\lab\)};
  \fill[selected] (0,0) circle (2.5pt)
    node[below right=-2pt] {\large $O$};
\end{tikzpicture}
\caption{Optimal configurations for $n\in\{6,7,8\}$
(for $n=6$, an equilateral triangle and the midpoints of its sides;
for $n=7$, the same together with its centre), and an optimal
configuration for $n\ge9$ (shown for $n=9$).}
    \label{fig:optimalConf}
\end{figure}
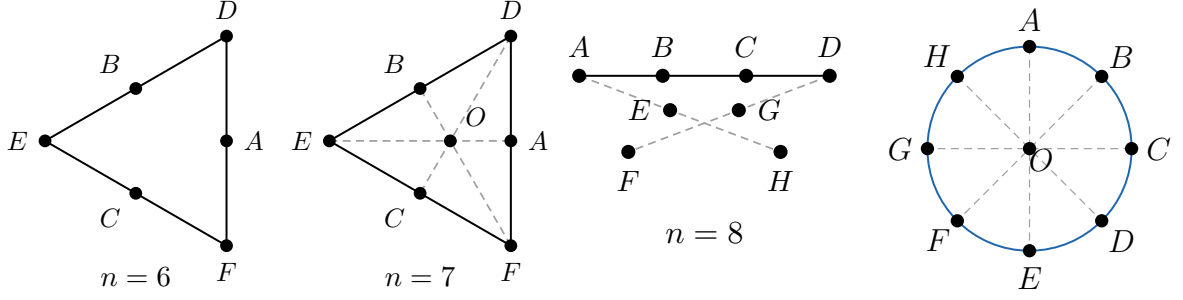

\begin{relatedremark}[Initial parallel investigations]
\Cref{thm:main} was obtained in the summer of 2026 by (at least) five investigations 
within a short period: those of Rafa\l{} Wrona, Liyan Wang, Stijn Cambie,
Denis Paliy, and the forum (on~\cite{Forum506}) user \texttt{mzn}. These investigations were independent, assuming that their AI use would not influence each other (which is unclear). This paper summarises the main insights from those.\\
The block-counting/inversion approach for eight points and the uniform
incidence approach for $n\ge15$ also appear in~\cite{Wrona}.
The seven-point Lemma~\ref{lem:seven} and the pointed
Lemma~\ref{lem:pointed} yield short versions of these arguments.
The carrier estimate, inversion count, and Melchior argument after
adjoining the inversion centre have counterparts in
\cite[Lemmas~2.1, 2.2 and 6.1]{Wang}; the proof in~\cite{Wang} treats
$n=15$ and $n=16$ separately. Paliy's parallel work~\cite{PaliyCorrected}
developed an exact computational and SAT-based approach to the finite
cases, e.g. $n=12,13$, with reproducible verification files.
\end{relatedremark}

\section{Proof for \texorpdfstring{$n=8$}{n=8}}\label{sec:n8}

\subsection{A seven-point lemma}

A line containing exactly two selected points is \emph{ordinary}.
The Sylvester--Gallai theorem states that every finite noncollinear real point set has such a line.
For completeness, choose a selected point $x$ and a determined line
$\ell$ not through $x$ for which the positive distance $d$ is minimal,
and let $h$ be the perpendicular foot from $x$ to $\ell$.
If $\ell$ contains three selected points, choose $A,B$ on the same
closed ray from $h$, with $A$ nearer. Then
$|AB|\le|hB|<|xB|$, and so
$\operatorname{dist}(A,xB)=d|AB|/|xB|<d$, a contradiction.

\begin{lemma}\label{lem:seven}
Seven real planar points with no four collinear determine at most six
three-point lines. If there are six, every point outside the set lies
on at most one ordinary line and at most one three-point line.
\end{lemma}

\begin{proof}
The seven points are not collinear, since no four are collinear.
Let $t_2,t_3$ count the ordinary and three-point lines. Pair counting gives
$t_2+3t_3=21$, and the ordinary-line theorem gives $t_2>0$, so $t_3\le6$.

At equality, the graph of ordinary pairs has three edges and only even
degrees: a point on $d$ three-point lines has ordinary degree $6-2d$.
A simple graph with three edges and all degrees even must be a triangle.
Hence the three ordinary edges form a triangle, whose supporting lines meet
pairwise at selected points.

The three-point degrees are therefore $2,2,2,3,3,3,3$. Since
$3\binom22+4\binom32=\binom62$, every pair of three-point lines also
meets at a selected point. Two ordinary lines cannot meet outside the selected set, and neither
can two three-point lines.
\end{proof}

\subsection{The lower bound}

Let $P$ be an admissible eight-point set, and let $C$ be its number of
determined circles. First suppose a line or circle contains exactly
$k\ge5$ selected points, and put $r=8-k$. Each outside point determines
distinct circles with all but at most $\lfloor k/2\rfloor$ pairs on the
carrier; only collinear triples fail (and this cannot occur for a line
carrier). Two such lists overlap in at most $\lfloor k/2\rfloor$ circles,
since the carrier pairs used by common circles are disjoint: otherwise
two distinct circles would share three points. Subtracting pairwise
overlaps and including the carrier circle gives the general estimate
\begin{equation}\label{eq:carrier-count}
 C\ge 1+r\left(\binom{k}{2}-\left\lfloor\frac{k}{2}\right\rfloor\right)
 -\binom{r}{2}\left\lfloor\frac{k}{2}\right\rfloor.
\end{equation}
For a line carrier there is no carrier circle to add, but every pair is
usable, so the same lower bound holds (as $r\lfloor k/2\rfloor\ge1$).
The right-hand side of \eqref{eq:carrier-count} is at least
$19>17$ for $k\in\{5,6,7\}$.
Bounds of this form also appear in the parallel works; see
\cite[Lemma~2.1]{Wang} and~\cite{Wrona,PaliyCorrected,WhiteClaude}.
We may henceforth assume that every line and circle contains at most
four selected points.

A \emph{block} is the full set of selected points on a line or circle
containing at least three of them. Let $b$ count four-point blocks and
$L$ count line blocks. Every triple belongs to exactly one block, so
the total number of blocks is $\binom83-3b$, and
\begin{equation}\label{eq:count}
 C=56-3b-L.
\end{equation}

For $p\in P$, let $r_p$ count the four-point blocks through $p$, and let
$s_p$ count the line blocks through $p$. Invert the other seven points
about $p$. Inversion takes
circles through $p$ to lines avoiding $p$, and preserves lines through
$p$. Thus the images (of the seven other points) have no four collinear, and their three-point
lines correspond precisely to the four-point blocks through $p$.
Lemma~\ref{lem:seven} gives $r_p\le6$.

Always $s_p\le3$, since these lines use disjoint subsets of at least
two other points each. If $r_p=6$, apply Lemma~\ref{lem:seven} to the
inversion center: at most one original three-point line and at most one
original four-point line pass through $p$, so $s_p\le2$. Consequently
$r_p+s_p\le8$ in every case. Summing gives
\begin{equation}\label{eq:bounds}
 4b=\sum_{p\in P}r_p\le48,
 \qquad
 4b+3L\le\sum_{p\in P}(r_p+s_p)\le64.
\end{equation}
For $b\le11$, equations~\eqref{eq:count} and~\eqref{eq:bounds} yield
$
 3C=168-9b-3L\ge104-5b\ge49>48,
$
so $C\ge17$.

If $b=12$, then $r_p=6$ everywhere. Each point lies on at most one
three-point line and at most one four-point line, so there are at most
two lines of either size. Two four-point lines would partition $P$ and
exclude every three-point line. Hence $L\le3$, and
\eqref{eq:count} gives $C=20-L\ge17$.

\subsection{Sharpness}

\begin{proof}[A construction]
Put $u=\sqrt7$ and take the $8$ points with coordinates $(x,y)$ as follows
\[
\begingroup
\setlength{\arraycolsep}{4.1pt}
\renewcommand{\arraystretch}{1.28}
\begin{array}{c|cccccccc}
      &A&B&C&D&E&F&G&H\\ \hline
x&-\tfrac32&-\tfrac12&\tfrac12&\tfrac32
  &\tfrac{1-u}{4}&-\tfrac{1+u}{4}
  &\tfrac{u-1}{4}&\tfrac{u+1}{4}\\
y&\tfrac u2&\tfrac u2&\tfrac u2&\tfrac u2
  &\tfrac{1+u}{4}&\tfrac{u-1}{4}
  &\tfrac{u+1}{4}&\tfrac{u-1}{4}
\end{array}
\endgroup
\]
The collinear triples are
$ABC,ABD,ACD,BCD,AEH,DFG$. The $11$ four-point circles are
$ABEF$, $ABGH$, $ACFG$, $ADEG$, $ADFH$, $BCEG$, $BCFH$, $BDEH$,
$CDEF$, $CDGH$, and $EFGH$. The remaining six circle blocks are
$ACE,ACH,BDF,BDG,BFG,CEH$.
A different eight-point configuration determining $17$ circles was
also reported by the forum user \texttt{mzn}~\cite{Forum506}.
The determinant criterion for
concyclicity, using $u^2=7$, verifies the incidences and the
distinctness of the circle equations.  Distinct circles share at most two points, so the listed circle blocks
account for disjoint sets of noncollinear triples. The checksum
$6+11\binom43+6=\binom83$ then confirms that the list is complete.
Thus the configuration determines seventeen circles. The
four-point circles are depicted in~\cref{fig:eight-rich}.
\end{proof}

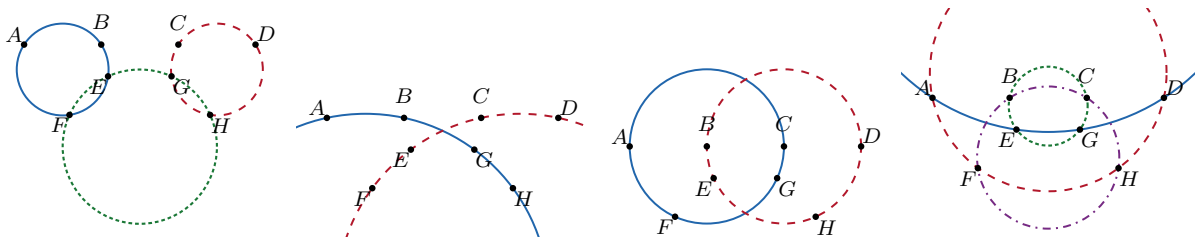
\begin{figure}[htbp]
\centering

\begin{tikzpicture}[scale=1.02]
  \clip (-1.9,-1.17) rectangle (1.9,1.72);
  \eightcoordinates
  \draw[rich1] (-1,1) circle[radius={sqrt(3-sqrt(7))}];
  \draw[rich2] (1,1) circle[radius={sqrt(3-sqrt(7))}];
  \draw[rich3] (0,0) circle[radius=1];
  \eightpoints
\end{tikzpicture}
\begin{tikzpicture}[scale=1.02]
  \clip (-1.9,-.22) rectangle (1.9,1.72);
  \eightcoordinates
  \draw[rich1] (-1,-1) circle[radius={sqrt(3+sqrt(7))}];
  \draw[rich2] (1,-1) circle[radius={sqrt(3+sqrt(7))}];
  \eightpoints
\end{tikzpicture}
\begin{tikzpicture}[scale=1.02]
  \clip (-1.9,.15) rectangle (1.9,2.48);
  \eightcoordinates
  \draw[rich1] (-.5,{sqrt(7)/2}) circle[radius=1];
  \draw[rich2] (.5,{sqrt(7)/2}) circle[radius=1];
  \eightpoints
\end{tikzpicture}
\begin{tikzpicture}[scale=1.02]
  \clip (-1.9,-.48) rectangle (1.9,2.48);
  \eightcoordinates
  \draw[rich1] (0,{1+sqrt(7)})
    circle[radius={sqrt(5+sqrt(7))}];
  \draw[rich2] (0,{sqrt(7)-1})
    circle[radius={sqrt(5-sqrt(7))}];
  \draw[rich3] (0,{(sqrt(7)+1)/3})
    circle[radius={sqrt(5-sqrt(7))/3}];
  \draw[rich4] (0,{(sqrt(7)-1)/3})
    circle[radius={sqrt(5+sqrt(7))/3}];
  \eightpoints
\end{tikzpicture}

\caption{The eleven four-point circles of the eight-point extremal
configuration.}
\label{fig:eight-rich}
\end{figure}

\section{Proof for \texorpdfstring{$n\ge15$}{n >= 15}}\label{sec:n_large}

For a noncollinear real point set, write $t_j$ for its number of
$j$-point spanned lines. Melchior's inequality~\cite{Melchior} is
\begin{equation}\label{eq:melchior}
 \sum_{j\ge2}(3-j)t_j\ge3.
\end{equation}
We also use Langer's incidence bound~\cite{Langer}: if the set has $m$
points and at most $2m/3$ are collinear, then by~\cite[Thm.~1.1]{deZeeuw}
\begin{equation}\label{eq:langer}
\sum_{j\ge2}jt_j\ge\left\lceil\frac{m(m+3)}3\right\rceil.
\end{equation}

\begin{lemma}\label{lem:pointed}
Let $Q$ be a noncollinear set of $m$ real points, let $o\notin Q$, and
let $\cL$ be the set of lines spanned by $Q$. With $j_\ell=|\ell\cap Q|$,
\[
 \sum_{\substack{\ell\in\cL\\o\notin\ell}}\frac1{j_\ell+1}
 \ge \frac1{12}\sum_{\ell\in\cL}j_\ell-\frac m9+\frac12.
\]
\end{lemma}
\begin{proof}
Put $I=\sum_{\ell\in\cL}j_\ell$,
$D=\sum_{\ell\in\cL}(3-j_\ell)$ and
$R=\sum_{\ell\in\cL,\,o\in\ell}(j_\ell+1)$.
Melchior's inequality on $Q$ gives $D\ge3$. Adjoining $o$ decreases the contribution
of each old line through $o$ by one. The lines through $o$ that
already contain two or more points of $Q$ partition those points into
disjoint subsets, and every other line joining $o$ to a point of $Q$
is new and ordinary. Thus $m-\sum_{o\in\ell}j_\ell$ ordinary
lines are created. Melchior on $Q\cup\{o\}$ gives $D+m-R\ge3$.
For every integer $j\ge2$,
$
 \frac1{j+1}-\frac{6-j}{12}
 =\frac{(j-2)(j-3)}{12(j+1)}\ge0$ and $\frac1{j+1}\le\frac{j+1}{9}.
$
Subtracting the weights of the lines through $o$ therefore yields
\begin{align*}
 \sum_{\substack{\ell\in\cL\\o\notin\ell}}\frac1{j_\ell+1}
 &=\sum_{\ell\in\cL}\frac1{j_\ell+1}
   -\sum_{\substack{\ell\in\cL\\o\in\ell}}\frac1{j_\ell+1}\\
 &\ge\sum_{\ell\in\cL}\frac{6-j_\ell}{12}
   -\sum_{\substack{\ell\in\cL\\o\in\ell}}\frac{j_\ell+1}{9}
  =\frac{I+2D}{12}-\frac R9\\
 &\ge\frac I{12}+\frac D{18}-\frac m9+\frac13
  \ge\frac I{12}-\frac m9+\frac12.
\end{align*}
Here $\sum_{\ell\in\cL}(6-j_\ell)=I+2D$ because
$3|\cL|=I+D$; we used $R\le D+m-3$ and then $D\ge3$.
\end{proof}

\begin{proof}[Proof for $n\ge15$]

The upper bound comes from a circle with $n-1$ selected points and its
center, as depicted in~\cref{fig:optimalConf} (rightmost) and explained before. 

Suppose first that a line or circle contains $k>n/2$ points, and put
$r=n-k\ge1$. The counting argument in~\eqref{eq:carrier-count} is valid
for every $n$, hence
\begin{equation}\label{eq:carrier}
 C(P)\ge 1+r\binom{k}{2}-\binom{r+1}{2}
 \left\lfloor\frac{k}{2}\right\rfloor
 \ge H_n(k):=1+\frac{(n-k)k(3k-n-3)}4.
\end{equation}
Put $G(n)=1+(n-1)(n-3)/2$, so that $g(n)=\lceil G(n)\rceil$.
On $[(n+1)/2,n-1]$ we have $H_n''(k)=(8n+6-18k)/4<0$, while
\[
 H_n(n-1)=G(n),\qquad
 H_n\!\left(\frac{n+1}{2}\right)-G(n)
 =\frac{(n-1)(n-3)(n-15)}{32}\ge0.
\]
Concavity and integrality give $C(P)\ge g(n)$.
We may now assume that every line and circle contains at most $n/2$
points of $P$. Invert $P\setminus\{p\}$ about each $p\in P$, obtaining
$Q_p$ of order $m=n-1$. It is noncollinear, since otherwise all of $P$
would lie on one line or circle. A $j$-point line of $Q_p$ avoiding $p$
corresponds to a circle containing exactly $j+1$ original points,
one of them $p$. Consequently the sum, over all $p$, of the left side
in Lemma~\ref{lem:pointed} is exactly $C(P)$.
Every line of $Q_p$ has at most $n/2-1\le2m/3$ points, so
\eqref{eq:langer} and the lemma give
\[
 C(P)\ge\frac n{36}
 \left(3\left\lceil\frac{(n-1)(n+2)}3\right\rceil-4n+22\right).
\]
At $n=15$ this is $505/6>84=g(15)-1$, so integrality suffices.
For $n\ge16$, dropping the ceiling and using $g(n)\le1+(n-2)^2/2$ gives
\[
 C(P)-g(n)\ge
 \frac{n(n^2-3n+20)}{36}-\left(1+\frac{(n-2)^2}{2}\right)
 =\frac{n(n-16)(n-5)+12(n-9)}{36}>0.
\]
This completes the lower bound.
\end{proof}


\section*{Tool and computational resource disclosure}

The mathematical investigations involved AI-assisted exploration and
checking, including use of OpenAI ChatGPT and Codex. GPT-6
assisted in comparing and streamlining the exposition, and Claude
was used to review computations, references and phrasing.
The authors checked and revised the mathematical arguments. A Lean~4 formalization is cited above.


\begingroup
\raggedright

\endgroup

\end{document}